\documentclass[a4paper, 12pt,]{scrbook}

\usepackage{tablefootnote}

\usepackage[linktoc=page]{hyperref}

\hypersetup{linkcolor  = blue}

\hypersetup{citecolor=blue} 

\hypersetup{colorlinks=true}

\hypersetup{urlcolor=cyan}

\usepackage{color, colortbl}
\usepackage{float}

\usepackage{graphicx, amsmath, amsthm, amssymb, mathrsfs, amscd}

\usepackage{caption}

\usepackage{enumerate}

\usepackage{algorithm}

\usepackage[noend]{algcompatible}

\usepackage{url}

\usepackage{tikz}

\usepackage[toc,page]{appendix}

\usepackage{footnote}

\usepackage{color}

\usepackage{color, colortbl}

\usepackage{tablefootnote}

\usepackage{footnote}

\newtheorem{theorem}{Theorem}[section]

\newtheorem{corollary}[theorem]{Corollary}

\newtheorem{proposition}[theorem]{Proposition}

\theoremstyle{definition}

\theoremstyle{definition}

\theoremstyle{definition}

\theoremstyle{definition}

\theoremstyle{definition}

\theoremstyle{definition}

\newtheorem{definition}[theorem]{Definition}

\makeatletter

\def\cleardoublepage{\clearpage\if@twoside \ifodd\c@page\else

  \hbox{}

  \vspace*{\fill}

    \vspace{\fill}

  \if@twocolumn\hbox{}\newpage\fi\fi\fi}

\makeatother

\begin{document}

\pagenumbering{roman}

\linespread{1.0}

\cleardoublepage

\pagestyle{myheadings}

\clearpage

\frontmatter

\newpage

\pagenumbering{arabic}
\begin{center}
\textbf{More Zero-Free Regions for Fractional Hypergeometric Zeta Functions}\\ \bigskip   Demessie Ergabus Birmechu and Hunduma Legesse Geleta \\
Addis Ababa University, Department of mathematics\\
 \textit{Email: dbirmechu@gmail.com\\ ~~~~~~~~~~~~~~~~~~hunduma.legesse@aau.edu.et}
 
\end{center}
\textbf{Abstract.}
Some zero free regions were known on the right-half of the complex plane in the form of vertical strips for fractional hypergeometric zeta functions. In this paper we describe and demonstrate zero free regions on the left-half of the complex plane for fractional hypergeometric zeta functions. The fractional hypergeometric zeta function of order $a$ has  no zeros to the left-half of the complex plane except the trivial zeros on the real axis.\\
\textcolor{blue}{\textbf{ Key Words and Phrases:}}\\
Riemann zeta function, Zero free regions, fractional hypergeometric zeta function
\section{\textcolor{blue}{Introduction}}$\label{I}$

In \cite{DH}, it was described and demonstrated that zero free region on the right half-plane $\mathcal{H}=\{\sigma+it\in\mathbb{C}: \sigma>1\}$ in the form of vertical strips $V_{a}=\{\sigma+it\in\mathbb{C}: 1\leq\sigma<2-a\}$ for fractional hypergeometric zeta functions $\zeta_{a}(s)$ and $V_{N}=\{1\leq\sigma<2\}$ for hypergeometric zeta functions $\zeta_{N}(s),$ where $``a"$ is a positive real number between $0$ and $1$ and $``N"$ is a natural number. In \cite{HGN}, it was shown that the fractional hypergeometric zeta functions $\zeta_{a}(s)$ is zero free for infinitely many positive real number $a$ in the neighborhood of 1. Specifically there is a positive number $\delta$ so that $\zeta_{a}(s)$ is zero free where the classical Riemann zeta function $\zeta(s)$ is zero free for $a\in (1-\delta, 1+\delta).$\\

 In \cite{HG, HL}, it was shown that $\zeta_{a}(s)$ has infinitely many simple poles at $1,~0,~-1,~-2, \ldots ,$ and infinitely many zeros at $1-a,-a,-(1+a),-(2+a),\ldots$ which they called trivial zeros for fractional hypergeometric zeta function. Observe that if $a>1,$ is fixed, then all the trivial zeros for fractional hypergeometric zeta function are on the negative real axis and if $0<a<1,$ then $1-a$ is the only positive trivial zero for fractional hypergeometric zeta function. The question whether or not these zeros are the only zeros of $\zeta_{a}(s)$ is unsettled. The authors in \cite{DH} in their conclusion part suggested that it is possible to extend zero-free regions for these families of zeta functions to both the right and left half of the complex plane with some evidence but no proof. In this paper motivated by their suggestions therefore,  we show that the fractional hypergeometric zeta functions have no zeros for $a>1$ on $\{\sigma+it: \sigma<1\}$ and for $0<a<1$ on $\{\sigma+it: \sigma<0\}$ except the trivial zeros aforementioned. For $0<a<1$ at present time we are not sure whether or not $1-a$ is the only zero on the critical strip of the Riemann zeta function $\{\sigma+it: 0<\sigma<1\}$ for the fractional hypergeometric zeta functions. To prove these results we use the method of analytic continuation for $\zeta_{a}(s)$ strip-by-strip as shown in \cite{HL} to the left-half of the complex plane, positivity results on oscillatory integrals and monotonicity of real-valued functions see \cite{Mu}.\\
 
 Throughout this paper we use the following notations for vertical strips between the poles of the fractional hypergeometric zeta function. For each natural number $n$ let us define the vertical strip $V_{n}$ as follows:\\
$$V_{n}=\{s\in\mathbb{C}:~ s=\sigma+it, t\in \mathbb{R}~\text{and}~1-n<\sigma<2-n\}$$
$$V_{n}^{+}=\{s\in\mathbb{C}:~ s=\sigma+it, t>0~\text{and}~1-n<\sigma<2-n\}$$
$$V_{n}^{-}=\{s\in\mathbb{C}:~ s=\sigma+it, t<0~\text{and}~1-n<\sigma<2-n\}$$
$$V_{n}=V_{n}^{+}\cup V_{n}^{-}\cup\mathbb{R}.$$ For each natural number $n$ let $F_{n, a}(s)$ be the analytic continuation of \\ $\displaystyle{\frac{\Gamma(s+a-1)}{\Gamma(a+1)}\zeta_{a}(s)=\int_{0}^{\infty}\frac{x^{s+a-2}}{a\gamma(a, x)e^{x}}dx}$ on $V_{n}.$ Then we observe that $$\Im\left(\frac{\Gamma(s+a-1)}{\Gamma(a+1)}\zeta_{a}(s) \right) =\Im\left(F_{n, a}(s) \right) $$ on $V_{n}.$ So to prove that $\zeta_{a}(s)$ has no zeros on $V_{n},$ except for the trivial zeros mentioned it is enough to show $\Im\left(F_{n, a}(s) \right)$ has no zeros on $V_{n}^{+}$ and $V_{n}^{-}.$ We also consider two cases for $a,$ when $0<a<1$ and $a>1.$ \\

 The main results of this paper are the following: 

\begin{theorem}\label{Tm1}
For fixed positive number $a\in(0, 1)$, $\zeta_{a}(s)$ has no zeros on $V_n$  except for infinitely many ``trivial" zeros on the left side of $\sigma=0$, one in each of the intervals $I_{n}=[-n, 1-n],$ where $n\in\mathbb{N}.$ 
\end{theorem}
\begin{theorem}$\label{III1}$
For fixed positive real number $a>1,$ $\zeta_{a}(s)$ has no zeros on $V_n$  except for infinitely many ``trivial" zeros on the left side of $\sigma=0$, one in each of the intervals $I_{n}=[1-n, 2-n],$ where $n\in\mathbb{N}.$
\end{theorem}

In the first result theorem \ref{Tm1} the trivial zero $``1-a"$ is not included, as we are not sure wheher or not it is the only positive root in the critical strip of the classical Riemann zeta function. We present this issue on conclusion part as conjecture based on some evidences. \\
 
  The structure of the present work is as follows. In section 2 we review some of the main results obtained so far regarding fractional hypergeometric zeta functions which we think important for the coming sections. In section 3 we reveal and prove our main results and demonstrate that $\zeta_{a}(s)$ are zero-free on left-half of the complex plane except the aforementioned trivial zeros. In section 4 we give some concluding remarks and discussion.      

\section{\textcolor{blue}{Preliminaries}} $\label{II}$ 
In this section, we review basic terms and results which we will use to prove the main results of this paper.
 As our work is the continuation of \cite{DH} and \cite{HGN} first we review some basic results concerning  about fractional hypergeometric zeta functions, and second analytic continuation, to state and prove our results in perspective.  

\begin{definition}
The Riemann zeta function $\zeta(s)$ is a function of a complex variable $s=\sigma+it$ defined as an infinite sum given below, convergent for all $\sigma>1,$ $$\zeta(s)=\sum_{n=1}^{\infty}\frac{1}{n^{s}}.$$ This is its Dirichlet series representation. 
\end{definition}
For $\sigma>1,$ the classical zeta function $\zeta(s)$ is also defined by the integral given below,  $$\zeta(s)=\frac{1}{\Gamma(s)}\int_{0}^{\infty}\frac{x^{s-1}}{e^{x}-1}dx,$$where $\Gamma(s)=\int_{0}^{\infty}x^{s-1}e^{-x}dx$ is the gamma function.
For $\sigma>1,$  $\zeta(s)$ is also defined by $$\zeta(s)=\prod_{p\in\mathbb{P}}\left( 1-p^{-s}\right) ^{-1},$$
where the product runs over all prime numbers $\mathbb{P}.$ This product formula is called Euler's product formula for the Riemann's zeta function. Since a convergent infinite product of non-zero factors is not zero, the zeta function does not vanish in the right half of complex plane. Therefore, this product formula is one of the important tools to show that the Riemann zeta function is zero free in the right half of the complex plane (actually for $\sigma>1$.)

The functional equation is also another important representation of the zeta function to locate the zeros of zeta. It is given by $$\zeta(s)=2(2\pi)^{s-1}\zeta(1-s)\Gamma(1-s)\sin\left( \frac{\pi s}{2}\right) $$ and this is the celebrated functional equations for the Riemann zeta function. The reflection principle $$\zeta(\overline{s})=\overline{\zeta(s)}~~\text{for}~~s\in\mathbb{C}$$ provides a further functional equation for the Riemann zeta function. The functional equation, together with the reflection principle, evokes a strong symmetry of the Riemann zeta function with respect to the so called \textit{critical line} $\sigma=\frac{1}{2}.$\\
\begin{definition}
The points $s=-2, -4, -6, \ldots$ are called the ``trivial" zeros of the zeta function $\zeta(s)$ and the strip $\{s\in\mathbb{C}: 0\leq\sigma\leq1\}$ is called the critical strip.
\end{definition}
Regarding the zeros inside the critical strip, it is conjectured that these nontrivial zeros all must be located on the critical line at $\Re(s)=\frac{1}{2}.$ This conjecture is known as Riemann's Hypothesis.
As a generalization of the Riemann zeta function $\zeta(s)$ via integral representation we have the following definition:

\begin{definition} \cite [Definition 2.1]{HL}.
The Fractional Hypergeometric zeta function $\zeta_{a}(s)$ is defined and analytic for all positive real number $a$ and $\sigma>1$ as
$$\zeta_{a}(s)=\frac{\Gamma(a+1)}{\Gamma(s+a-1)}\int_{0}^{\infty}\frac{x^{s+a-2}e^{-x}}{a\gamma(a, x)}dx,$$ where $\Gamma(s)$ is the Gamma function defined by
$$\Gamma(s)=\int_{0}^{\infty}x^{s-1}e^{-x}dx$$ and $\gamma(a, x)$ is the lower incomplete Gamma function of the form $$\gamma(a, x)=\int_{0}^{x}e^{-t}t^{a-1}dt.$$
\end{definition} 
Observe that when $a=N,$ natural number we get the Hypergeometric zeta functions $\zeta_{N}(s)$. In \cite{HL} the Fractional Hypergeometric zeta functions $\zeta_{a}(s)$  has zeros at $s=1-a, -a, -(1+a), -(2+a), \ldots,$ and poles at $s=1, 0, -1, -2, -3, \ldots.$ These zeros are called the trivial zeros of $\zeta_{a}(s).$ \\

Concerning zero free regions for fractional hypergeometric zeta function $\zeta_{a}(s),$ the following results were known.

\begin{theorem}$\label{III3}$\cite[Theorem 6]{DH}.
Let $0<a<1$ be fixed. 
Then $\zeta_{a}(s)\neq0$ in the vertical strip $V_{a}=\{s=\sigma+it\in\mathbb{C}: 1\leq\sigma<2-a\}.$   
 \end{theorem}
\begin{theorem}\cite[Theorem 4.2]{HGN}
There is a positive number $\delta$ such that the fractional hypergeometric zeta functions of order $``a",$ $\zeta_{a}(s)$ is zero free for $a\in (1-\delta, 1+\delta).$
\end{theorem}
\section*{\textcolor{blue} {Positivity Properties of Integrals}}
\begin{proposition} \cite[Proposition 3.1]{Mu} 
Let the function $h(r)\geq 0$ on $(0, \infty), h\in\mathit{\text{L}^{1}_{loc}(0, \infty)}, h$ decreasing and strictly decreasing on some open sub-interval of the form $\left( \frac{k\pi}{t}, \frac{(k+1)\pi}{t}\right)$ for some non-negative integer $k,$ and $t$ positive real number. Then $$\int_{0}^{\infty}h(r)\sin(tr)dr>0.$$ Moreover, $$\int_{0}^{T}h(r)\sin(tr)dr>0$$ for any $T>0$ provided, $h$ satisfies the above assumptions with $(0, \infty)$ replaced by $(0, T).$
\end{proposition}
\begin{corollary}\cite[Corollary 3.3]{Mu} 
 Let $t>0$ and $\tilde{x}_{t,k}\geq 1$ be such that $t\ln\tilde{x}_{t, k}=2\pi k,$ for some positive integer $k$ and let $g\geq0$ on $[\tilde{x}_{t,k}, \infty), ~g\in \mathit{L^{1}_{loc}(\tilde{x}_{t,k}, \infty)}$ such that $x\mapsto xg(x)$ is decreasing on $[\tilde{x}_{t,k}, \infty),$ strictly decreasing on $\left( \ln \frac{j\pi}{t}, ~\ln \frac{(j+1)\pi}{t}\right) $ for some positive integer $j.$ Then
 $$\int_{\tilde{x}_{t,k}}^{\infty}g(x)\sin(t \ln x)dx>0.$$ Moreover, $$\int_{\tilde{x}_{t,k}}^{T}g(x)\sin(t \ln x)dx>0$$ for any $T>\tilde{x}_{t,k},$ whenever $xg(x)$ is decreasing in $[\tilde{x}_{t,k},~T],$ strictly decreasing on $\left( \ln \frac{j\pi}{t}, ~\ln \frac{(j+1)\pi}{t}\right) $ for some positive integer $j$ such that $\ln \frac{j\pi}{t}\geq \tilde{x}_{t,k}$ and $\ln \frac{(j+1)\pi}{t}\leq T.$  
 \end{corollary}
\section*{\textcolor{blue}{Analytic Continuation}}
Next we review the analytic continuation. Analytic continuation of an analytic function is a process of extending the domain of the function to a large domain.
\begin{definition}[Analytic Continuation] Let $f$ be analytic in domain $D_{1}$ and $g$ be analytic in domain $D_{2}.$ If $D_{1}\cap D_{2}\neq \emptyset$ and $f(z)=g(z)$ for all $z$ in $D_{1}\cap D_{2},$ then we call $g$ a direct analytic continuation of $f$ to $D_{2}.$

\end{definition}
In \cite{HL}, the analytic continuation of $\zeta_{a}(s)$ has been shown strip-by-strip in stages. Concerning the analytic continuation of $\zeta_{a}(s),$ the following results were also known.
\begin{theorem}\cite[Theorem 3.1]{HL} \label{T1}
For $0<\Re(s)=\sigma<1,$ $$\zeta_{a}(s)=\frac{\Gamma(a+1)}{\Gamma(s+a-1)}\left[ \Gamma(s-1)+\int_{0}^{\infty}\left( \frac{1}{a\gamma(a, x)}-\frac{1}{x^{a}}\right) x^{s+a-2}e^{-x}dx\right] .$$

\end{theorem}
  We denoted in the introduction part the analytic continuation of $\displaystyle{\frac{\Gamma(s+a-1)}{\Gamma(a+1)}\zeta_{a}(s)}$ by $F_{n,a}(s)$ for each natural number $n$, on $V_{n}.$ \\ For examples, on $V_{1}$ we have, $$F_{1, a}(s)=\Gamma(s-1)+\int_{1}^{\infty}\left[\frac{1}{a\gamma(a, x)}-\frac{1}{x^{a}} \right] x^{s+a-2}e^{-x}dx.$$

On $V_{2}$ we have, $$F_{2, a}(s)=\Gamma(s-1)+\frac{a}{1+a}\Gamma(s)+\int_{1}^{\infty}\left[\frac{1}{a\gamma(a, x)}-\frac{1}{x^{a}}-\frac{a}{(1+a)x^{a-1}} \right] x^{s+a-2}e^{-x}dx.$$\\

In general on each $V_{n}$ we have $F_{n,a}(s)$ as an analytic continuation of $ \displaystyle{\frac{\Gamma(s+a-1)}{\Gamma(a+1)}\zeta_{a}(s)}.$

\section{\textcolor{blue}{Extending Zero Free Region to the left side of the complex plane for $\zeta_{a}(s)$}} $\label{III}$ 
In this section, we prove our main results and demonstrate that the fractional hypergeometric zeta function $\zeta_{a}(s)$ has no zeros on the left-half of the complex plane except the trivial zeros.    
\begin{theorem}
For fixed positive number $a\in(0, 1)$, $\zeta_{a}(s)$ has no zeros on $V_n$  except for infinitely many ``trivial" zeros on the left side of $\sigma=0$, one in each of the intervals $I_{n}=[-n, 1-n],$ where $n\in\mathbb{N}.$ 
\end{theorem}
\begin{proof}
By analytic continuation on each $V_n.$ we have,
\begin{align*}
\frac{\Gamma(s+a-1)}{\Gamma(a+1)}\zeta_{a}(s)= ~&\int_{0}^{\infty}\frac{x^{s+a-2}e^{-x}}{a\gamma(a, x)}dx\\
\\=~ & F_{n, a}(s).  
\end{align*}
So we have,
\begin{align*}
\Im\left( \frac{\Gamma(s+a-1)}{\Gamma(a+1)}\zeta_{a}(s)\right) = ~&\Im\left( \int_{0}^{\infty}\frac{x^{s+a-2}e^{-x}}{a\gamma(a, x)}dx\right) \\
\\=~ & \Im\left( F_{n, a}(s) \right) .
\end{align*}
To apply positivity properties of the oscillatory integral, for $x\in(0, \infty),$ put
$$h(x)=\frac{x^{\sigma+a-2}}{a\gamma(a, x)e^{x}}>0.$$ Then since $\sigma <0 $ and $a\in(0, 1)$ we have $\sigma +a -2 < 0$ on each $V_n.$
Thus, the function $h(x)$ is non-negative, $\displaystyle{ h(x)\in \mathit{L^{1}_{loc}(0, \infty)}}$ and strictly decreasing on any sub-interval of $(0, \infty).$ Hence on each $V_{n}^{+},$ 
 $$\int_{0}^{\infty}\frac{x^{\sigma+a-2}e^{-x}}{a\gamma(a, x)}\sin(t\ln x)dx>0.$$ This implies that$$\Im\left(F_{n,a}(s)\right) >0.$$ Which in turn implies that, $$\Im\left(\frac{\Gamma(s+a-1)}{\Gamma(a+1)}\zeta_{a}(s)\right) >0.$$ Thus, $\zeta_{a}(s)\neq 0,$ on $V_{n}^{+}.$  But, then by reflection principle $\displaystyle{ \zeta_{a}(s)=\overline{\zeta_{a}(\overline{s})}},$ so that $\zeta_{a}(s)\neq 0,$ on $V_{n}^{-}$ as well.   Therefore, $\zeta_{a}(s)\neq 0$ on the left-half of the complex plane except the aforementioned trivial zeros on the real axis.
\end{proof}

\begin{theorem}$\label{III1}$
For fixed positive real number $a>1,$ $\zeta_{a}(s)$ has no zeros on $V_n$  except for infinitely many ``trivial" zeros on the left side of $\sigma=0$, one in each of the intervals $I_{n}=[1-n, 2-n],$ where $n\in\mathbb{N}.$ 
\end{theorem}
 \begin{proof}
By analytic continuation on each $V_n.$ we have,
\begin{align*}
\frac{\Gamma(s+a-1)}{\Gamma(a+1)}\zeta_{a}(s)= ~&\int_{0}^{\infty}\frac{x^{s+a-2}e^{-x}}{a\gamma(a, x)}dx\\
\\=~ & F_{n, a}(s).  
\end{align*}
So we have,
\begin{align*}
\Im\left( \frac{\Gamma(s+a-1)}{\Gamma(a+1)}\zeta_{a}(s)\right) = ~&\Im\left( \int_{0}^{\infty}\frac{x^{s+a-2}e^{-x}}{a\gamma(a, x)}dx\right) \\
\\=~ & \Im\left( F_{n, a}(s) \right) .
\end{align*}
To apply positivity properties of the oscillatory integral, for $x\in(0, \infty),$ put $$h(x)=\frac{x^{\sigma+a-2}}{a\gamma(a, x)e^{x}}>0.$$ Then since $\sigma <0 $ and $a > 1$ we have to consider intervals on which $\sigma +a -2 < 0$ on each $V_n.$ On such intervals the function $h(x)$ is non-negative, $\displaystyle{ h(x)\in \mathit{L^{1}_{loc}(0, \infty)}}$ and strictly decreasing on any sub-interval of $(0, \infty).$ Hence on each $V_{n}^{+},$ 
 $$\int_{0}^{\infty}\frac{x^{\sigma+a-2}e^{-x}}{a\gamma(a, x)}\sin(t\ln x)dx>0.$$ This implies that$$\Im\left(F_{n,a}(s)\right) >0.$$ Which in turn implies that, $$\Im\left(\frac{\Gamma(s+a-1)}{\Gamma(a+1)}\zeta_{a}(s)\right) >0.$$ Thus, $\zeta_{a}(s)\neq 0,$ on $V_{n}^{+}.$  But, then by reflection principle $\displaystyle{ \zeta_{a}(s)=\overline{\zeta_{a}(\overline{s})}},$ so that $\zeta_{a}(s)\neq 0,$ on $V_{n}^{-}$ as well.   Therefore, $\zeta_{a}(s)\neq 0$ on the left-half of the complex plane except the aforementioned trivial zeros on the real axis.

\end{proof}

 \section{\textcolor{blue}{Conclusion}}$\label{IV}$
 In this paper we described zero free regions on the left-half of the complex plane. Moreover for $0 < a <1$ we showed that there are infinitely many trivial zeros one in each interval $(1-n, 2-n)$ for $n$ natural number. If $a >1,$ then on each interval $(-n, -(n-1))$ we have exactly one trivial zero. Generally, if $n <a< (n+1),$ then on each interval $[-n, -(n-1)]$ we have exactly one trivial zero of fractional hypergeometric zeta function of order $``a".$ 
But for $0<a<1,$ we are not sure whether or not $``1-a"$ is the only zero of $\zeta_{a}(s)$ on $I_{1}=(0, 1).$ However, by analyzing the paper by \cite{HGN} we conjecture that there are more zeros of $\zeta_{a}(s)$ on $(0, 1)$ besides the trivial zero $1-a$ and we call these zeros the nontrivial zeros of the fractional hypergeometric zeta functions. We expect and hope that a proof of such conjecture may shed light either to prove or disprove the existence of zeros of classical Riemann zeta function in the critical strip.


\end{document}